\documentclass[reqno]{amsart}

\usepackage[utf8]{inputenc}
\usepackage[T1]{fontenc}
\usepackage{amsthm}
\usepackage{amsmath}
\usepackage{amssymb}
\usepackage{enumitem}
\usepackage[dvips]{graphicx}
\usepackage{color}
\usepackage{hyperref}
\usepackage{url}
\usepackage[left=3.5cm, right=3.5cm, paperheight=11.5in]{geometry}
\usepackage{fancyhdr}
\usepackage{mathrsfs}
\usepackage{stmaryrd}
\usepackage{cite}
\usepackage{bm}

\setlist[description]{leftmargin=0cm,labelindent=0cm}

\newtheorem{theorem}{Theorem}
\newtheorem{corollary}[theorem]{Corollary}
\newtheorem*{corollary*}{Corollary}
\newtheorem{lemma}{Lemma}
\newtheorem*{lemma*}{Lemma}
\newtheorem*{zsigmondy*}{Zsigmondy's theorem}
\newtheorem*{main*}{Main Theorem}

\theoremstyle{remark}

\hypersetup{
    pdftitle={On the number of distinct prime factors of a sum of super-powers},
    pdfauthor={Paolo Leonetti and Salvatore Tringali},
    pdfmenubar=false,
    pdffitwindow=true,
    pdfstartview=FitH,
    colorlinks=true,
    linkcolor=blue,
    citecolor=green,
    urlcolor=cyan
}

\providecommand{\NNb}{\mathbf{N}}
\providecommand{\PPc}{\mathcal{P}}
\providecommand{\QQb}{\mathbf{Q}}
\providecommand{\QQc}{\mathcal{Q}}
\providecommand{\RRb}{\mathbf{R}}
\providecommand{\ZZb}{\mathbf{Z}}
\providecommand\llb{\llbracket}
\providecommand\rrb{\rrbracket}

\newcommand{\slog}{\mathrm{slog}}

\newcommand{\fixed}[2][1]{%
  \begingroup
  \spaceskip=#1\fontdimen2\font minus \fontdimen4\font
  \xspaceskip=0pt\relax
  #2%
  \endgroup
}

\begin{document}
\title{On the number of distinct prime factors \\ of a sum of super-powers}

\author{Paolo Leonetti}
\address{Department of Decision Sciences, Universit\`a L. Bocconi, via Roentgen 1, 20136 Milano, Italy}
\email{leonetti.paolo@gmail.com}
\urladdr{https://sites.google.com/site/leonettipaolo/}

\author{Salvatore Tringali}
\address{Department of Mathematics, Texas A\&M University at Qatar \\ PO Box 23874 Doha, Qatar}
\curraddr{Institute for Mathematics and Scientific Computing, University of Graz, NAWI Graz | Heinrichstr. 36, 8010 Graz, Austria}
\email{salvatore.tringali@uni-graz.at}
\urladdr{http://imsc.uni-graz.at/tringali}




\thanks{P.L. was supported by a PhD scholarship from Universit\`a ``Luigi Bocconi'', and S.T. by NPRP grant No. [5-101-1-025] from the Qatar National Research Fund (a member of Qatar Foundation), by the French ANR Project No. ANR-12-BS01-0011, and by the Austrian FWF Project No. M 1900-N39.}

\begin{abstract}
Given $k, \ell \in {\bf N}^+$, let $x_{i,j}$ be, for $1 \le i \le k$ and $0 \le j \le \ell$, some fixed integers, and define, for every $n \in {\bf N}^+$, $s_n := \sum_{i=1}^k \prod_{j=0}^\ell x_{i,j}^{n^j}$. We prove that the following are equivalent:
\begin{enumerate}[label={\rm (\alph{*})}]
\item There are a real $\theta > 1$ and infinitely many indices $n$ for which the number of distinct prime factors of
$s_n$ is greater than the super-logarithm of $n$ to base $\theta$.
\item There do not exist non-zero integers $a_0,b_0,\ldots,a_\ell,b_\ell $ such that $s_{2n}=\prod_{i=0}^\ell a_i^{(2n)^i}$ and $s_{2n-1}=\prod_{i=0}^\ell b_i^{(2n-1)^i}$ for all $n$.
\end{enumerate}
We will give two different proofs of this result, one based on a theorem of Evertse
(yielding, for a fixed finite set of primes $\mathcal S$, an effective bound on the number of non-degenerate solutions of an $\mathcal S$-unit equation in $k$ variables over the rationals) and the other using only elementary methods.

As a corollary, we find that, for fixed $c_1, x_1, \ldots,c_k, x_k \in \mathbf N^+$, the number of distinct prime factors of $c_1 x_1^n+\cdots+c_k x_k^n$ is bounded, as $n$ ranges over $\mathbf N^+$, if and only if $x_1=\cdots=x_k$.
\end{abstract}

\maketitle
\thispagestyle{empty}
\section{Introduction}\label{sec:introduction}
Given $k, \ell \in {\bf N}^+$, let $x_{i,j}$ be, for $1 \le i \le k$ and $0 \le j \le \ell$, some fixed rationals.
Then, consider the $\QQb$-valued sequence $(s_n)_{n\ge 1}$ obtained by taking
\begin{equation}\label{equ:sequence_type_I}
s_n := \sum_{i=1}^k \prod_{j=0}^\ell x_{i,j}^{n^j}
\end{equation}
for every $n \in \NNb^+$ (notations and terminology, if not explained, are standard or should
be clear from the context); we refer to $s_n$ as a sum of super-powers of degree $\ell$.
Notice that $(s_n)_{n \ge 1}$ includes as a special case any $\QQb$-valued sequence of general term
\begin{equation}
\label{equ:sequence_type_II}
\sum_{i = 1}^k \prod_{j = 1}^{\ell_i} y_{i,j}^{f_{i,j}(n)},
\end{equation}
where, for each $i = 1, \ldots, k$, we let $\ell_i \in \NNb^+$ and $y_{i,1}, \ldots, y_{1,\ell_i} \in \QQb \setminus \{0\}$, while $f_{i,1}, \ldots, f_{i,\ell_i}$ are polynomials in one variable with integral coefficients. Conversely,  sequences of the form \eqref{equ:sequence_type_I} can be viewed as sequences of the form \eqref{equ:sequence_type_II}, the latter being prototypical of scenarios where polynomials are replaced with more general functions $\NNb^+ \to \ZZb$ (see also \S{ }\ref{sec:closings}).

Now, let $\omega(x)$ denote, for each $x \in \ZZb \setminus \{0\}$, the number of distinct prime divisors of $x$, and define $\omega(0):=\infty$. Then, for $x \in \ZZb$ and $y \in \NNb^+$ we set $\omega(xy^{-1}) := \omega(\delta^{-1} x) + \omega(\delta^{-1} y)$, where $\delta$ is the greatest common divisor of $x$ and $y$.

In addition, given $n \ge 2$ and $\theta > 1$, we write $\slog_\theta(n)$ for the super-logarithm of $n$ to base $\theta$, that is, the largest integer $\kappa \ge 0$ for which $\theta^{\otimes \kappa} \le n$, where $\theta^{\otimes 0} := 1$ and $\theta^{\otimes \kappa} := \theta^{\theta^{\otimes (\kappa-1)}}$ for $\kappa \ge 1$;
note that $\slog_\theta(n) \to \infty$ as $n \to \infty$.

The main goal of this paper is to provide necessary and sufficient conditions for the boundedness of the sequence $(\omega(s_n))_{n \ge 1}$. More precisely, we have:

\begin{theorem}\label{th:main}
The following are equivalent:
\begin{enumerate}[label={\rm (\alph{*})}]
\item\label{it:main-theorem(b)} There is a base $\theta > 1$ such that $\omega(s_n)>\slog_\theta(n)$ for infinitely many $n$.
\item\label{it:main-theorem(c)} $\limsup_{n \to \infty} \omega(s_n) = \infty$.
\item\label{it:main-theorem(d)} There do not exist non-zero rationals $a_0,b_0,\ldots,a_\ell,b_\ell$ such that $s_{2n}=\prod_{j=0}^\ell a_j^{(2n)^j}$ and $s_{2n-1}=\prod_{j=0}^\ell b_j^{(2n-1)^j}$ for all $n$.
\end{enumerate}
\end{theorem}
We will give two proofs of Theorem \ref{th:main} in \S{ }\ref{sec:proof}, one based on a theorem of Evertse (yielding, for a fixed finite set of primes $\mathcal S$, an effective bound on the number of non-degenerate solutions of an $\mathcal S$-unit equation in $k$ variables over the rationals), and the other using only elementary methods: It is, in fact, in the second proof that there lies, we hope, the added value of this work.

Results in the spirit of Theorem \ref{th:main} have been obtained by various authors in the special case of $\ZZb$-valued sequences raising from the solution of non-degenerate linear homogeneous recurrence equations with (constant) integer coefficients of order $\ge 2$, namely, in relation to a sequence $(u_n)_{n \ge 1}$ of general term
\begin{equation}
\label{equ:stewarts_equ}
u_n := \alpha_1^n f_1(n)  + \cdots + \alpha_h^n f_h(n),
\end{equation}
where $\alpha_1, \ldots, \alpha_h$ are the non-zero (and pairwise distinct) roots of the characteristic polynomial of the recurrence under consideration, and $f_1, \ldots, f_h$ are non-zero polynomials in one variable with coefficients in the smallest field extension of the rational field containing $\alpha_1, \ldots, \alpha_h$, see \cite[Theorem C.1]{ShTij}.
(A recurrence is non-degenerate if its characteristic polynomial has at least two distinct non-zero complex roots and the ratio of any two distinct non-zero characteristic roots is not a root of unity.)
More specifically, it was shown by van der Poorten and Schlickewei \cite{vanderP82} and, independently, by Evertse \cite[Corollary 3]{evertse84}, using Schlickewei's $p$-adic analogue of Schmidt's Subspace Theorem \cite{schli76}, that the greatest prime factor of $u_n$ tends to $\infty$ as $n \to \infty$.

In a similar note, effective lower bounds on the greatest prime divisor and on the greatest square-free factor of a sequence of type \eqref{equ:stewarts_equ} were obtained under mild assumptions by Shparlinski \cite{Shparli80} and Stewart \cite{Stewart1, Stewart3, Stewart2}, based on variants of Baker's theorem on linear forms in the logarithms of algebraic numbers \cite{BaWu}. Further results in the same spirit can be found in \cite[\S{ }6.2]{EPSW03}.

On the other hand, Luca has shown in \cite{Lu} that if $(v_n)_{n \ge 1}$ is a sequence of rational numbers satisfying a recurrence of the form
$$
g_0(n) v_{n+2} + g_1(n) v_{n+1} + g_2(n) v_n = 0, \quad\text{for all }n \in \NNb^+,
$$
where $g_0$, $g_1$ and $g_2$ are univariate polynomials over the rational field and not all zero, and $(v_n)_{n \ge n_0}$ is not binary recurrent (viz., a solution of a linear homogeneous second-order recurrence equation with integer coefficients) for some $n_0 \in \NNb^+$, then there exists a real constant $c > 0$ such that the product of the numerators and denominators (in the reduced fraction) of the non-zero rational terms of the finite sequence $(v_i)_{1 \le i \le n}$ has at least $c \log n$ prime factors as $n \to \infty$.

Lastly, it seems worth noting that Theorem \ref{th:main} can be significantly improved in special cases. E.g., given $a, b \in \mathbf N^+$ with $a \ne b$, we have by Zsigmondy's theorem \cite{Zsig} that $\omega(n) \ge d(n)-2$ for all $n$, where $d(n)$ is the number of (positive integer) divisors of $n$. Now, it is known, e.g., from \cite{SmSu} that $\frac{1}{n} \sum_{i=1}^n d(i)$ is asymptotic to $\log n$ as $n \to \infty$. So, it follows that there exist a constant $c \in \RRb^+$ and infinitely many $n$ for which $\omega(a^n-b^n)>c\log n$.
\begin{corollary}
\label{cor:(2)}
The sequence $(\omega(s_n))_{n \ge 1}$ is bounded if and only if there exist non-zero rationals $a_0,b_0,\ldots,a_\ell,b_\ell$ such that $s_{2n}=\prod_{j=0}^\ell a_j^{(2n)^j}$ and $s_{2n-1}=\prod_{j=0}^\ell b_j^{(2n-1)^j}$ for all $n$.
\end{corollary}
\begin{corollary}\label{corollary}
Let $c_1, \ldots, c_k \in \QQb^+$ and $x_1,\ldots,x_k \in \QQb \setminus \{0\}$. Then, $(\omega(c_1 x_1^n+\cdots+c_kx_k^n))_{n\ge 1}$ is a bounded sequence only if $|x_1|=\cdots=|x_k|$, and this condition is also sufficient provided that $\sum_{i=1}^k \varepsilon_i c_i \ne 0$, where, for each $i \in \llb 1, k \rrb$, $\varepsilon_i := x_i \cdot |x_i|^{-1}$ is the sign of $x_i$.
\end{corollary}
The proof of Corollary \ref{corollary} is postponed to \S{ }\ref{sec:proof:cor(3)}, while Corollary \ref{cor:(2)} is trivial by Theorem \ref{th:main}.
\subsection*{Notations} We reserve the letters $h$, $i$, $j$, and $\kappa$ (with or without subscripts) for non-negative integers, the letters $m$ and $n$ for positive integers, the letters $p$ and $q$ for (positive rational) primes, and the letters $A$, $B$, and $\theta$ for real numbers. We denote by $\mathbf P$ the set of all (positive rational) primes and by $\upsilon_p(x)$, for $p \in \mathbf P$ and a non-zero $x \in \mathbf Z$, the $p$-adic valuation of $x$, viz., the exponent of the largest power of $p$ dividing $x$. Given $X \subseteq \RRb$, we take $X^+ := X \cap {]0,\infty[}\fixed[0.2]{\text{ }}$. Further notations, if not explained, are standard or should be clear from the context.
%
%
\section{Proof of Theorem \texorpdfstring{\ref{th:main}}{1}}
\label{sec:proof}
The implications \ref{it:main-theorem(b)} $\Rightarrow$ \ref{it:main-theorem(c)} and \ref{it:main-theorem(c)} $\Rightarrow$ \ref{it:main-theorem(d)} are straightforward, and \ref{it:main-theorem(d)} $\Rightarrow$ \ref{it:main-theorem(b)} is trivial if at least one of the sequences $(s_{2n})_{n \ge 1}$ and $(s_{2n-1})_{n \ge 1}$ is eventually zero.

Therefore, we can just focus on the two cases below, in each of which we have to prove that there exists a base $\theta > 1$ such that $\omega(s_n) > \slog_\theta(n)$ for infinitely many $n$.
\vskip 0.3cm
\noindent{}\textbf{Case (i):} \textit{There do not exist $a_0, \ldots, a_\ell \in \QQb$ such that $s_{2n} = \prod_{j=0}^\ell a_j^{(2n)^j}$ for all $n$.} Then $k \ge 2$, $s_n \ne 0$ for infinitely many $n$, and $|x_{i,j}| \ne 1$ for some $(i,j) \in \llb 1, k \rrb \times \llb 1, \ell \rrb$ (otherwise we would have $s_{2n} = \sum_{i=1}^k x_{i,0}$, a contradiction).

Without loss of generality, we can suppose that $x_{i,j} \ne 0$ for all $(i,j) \in \llb 1, k \rrb \times \llb 0, \ell \rrb$ (otherwise we end up with a sum of super-powers with fewer than $k$ summands), and actually that $x_{i,j} > 0$ for $j \ne 0$: This is because $\prod_{j=0}^\ell x_{i,j}^{(2n)^j} = x_{i,0} \cdot \prod_{j=1}^\ell |x_{i,j}|^{(2n)^j}$ for all $n$, and, insofar as we deal with Case (i), we can replace $(s_n)_{n \ge 1}$ with the subsequence $(s_{2n})_{n \ge 1}$, after noticing that $\omega(s_{2n}) > \text{slog}_\theta(n)$, for some $\theta > 1$, only if $\omega(s_{2n}) > \text{slog}_{2\theta}(2n)$, which is easily proved by induction (we omit details). Accordingly, we may also assume
\begin{equation}
\label{equ:ordered}
(x_{1,1},\ldots,x_{1,\ell}) \prec \cdots \prec (x_{k,1},\ldots,x_{k,\ell}),
\end{equation}
where $\prec$ denotes the binary relation on $\RRb^\ell$ defined by taking $(u_1, \ldots, u_\ell) \prec (v_1, \ldots, v_\ell)$ if and only if $|u_i| < |v_i|$ for some $i \in \llb 1, \ell \rrb$ and $|u_j| = |v_j|$ for $i < j \le \ell$ (the $\ell$-tuples $(x_{i,1},\ldots,x_{i,\ell})$ cannot be equal to each other for all $i \in \llb 1, k \rrb$, and on the other hand, if two of these tuples are equal, then we can add up some terms in \eqref{equ:sequence_type_I} so as to obtain a sum of super-powers of degree $\ell$, but again with fewer summands). It follows by \eqref{equ:ordered} that there exists $N \in \NNb^+$ such that
\begin{equation}
\label{equ:subsums}
\sum_{i \in I} \prod_{j=0}^\ell x_{i,j}^{n^j} \ne 0, \quad\text{for all } n \ge N \text{ and } \emptyset \ne I \subseteq \llb 1, k \rrb.
\end{equation}
Now, for each $(i,j) \in \llb 1, k \rrb \times \llb 0, \ell \rrb$ pick $\alpha_{i,j}, \beta_{i,j} \in \ZZb$ such that $\alpha_{i,j} > 0$ and $x_{i,j} = \alpha_{i,j}^{-1} \beta_{i,j}$, and consequently set
$\tilde{x}_{i,j} := \alpha_j x_{i,j}$, where $\alpha_j := \alpha_{1,j} \cdots \alpha_{k,j}$; note that $\tilde{x}_{i,j}$ is a non-zero integer, and $\tilde{x}_{i,j} > 0$ for $j \ne 0$. Then, let $
u_n := \sum_{i = 1}^k \prod_{j=0}^\ell \tilde{x}_{i,j}^{n^j}$
and $v_n := \prod_{j=0}^\ell \alpha_j^{\fixed[0.15]{\text{ }}n^j}$, so that $s_n = u_n v_n^{-1}$.

Clearly, $(u_n)_{n \ge 1}$ and $(v_n)_{n \ge 1}$ are integer sequences, and $(\tilde{x}_{i,1},\ldots,\tilde{x}_{i,\ell}) \prec (\tilde{x}_{j,1},\ldots,\tilde{x}_{j,\ell})$ for $1\le i < j \le k$. Moreover,
$\omega(s_n) \ge \omega(u_n) - \omega(v_n) = \omega(u_n) - \omega(v_1)$ for all $n$. This shows that it is sufficient to prove the existence of a base $\theta > 1$ such that $\omega(u_{n}) > \text{slog}_\theta(n)$ for infinitely many $n$, and it entails, along with the rest, that we can further assume that $x_{i,j}$ is a non-zero \textit{integer} for every $(i,j) \in \llb 1, k \rrb \times \llb 0, \ell \rrb$.

We claim that it is also enough to assume $\delta_0 = \cdots = \delta_\ell = 1$, where for each $j \in \llb 0, \ell \rrb$ we let $\delta_j := \gcd(x_{1,j},\ldots,x_{k,j})$.
In fact, define, for $1 \le i \le k$ and $0 \le j \le \ell$, $\xi_{i,j} := \delta_j^{-1}x_{i,j}$, and let $(w_n)_{n \ge 1}$ and $(\tilde{s}_n)_{n \ge 1}$ be the integer sequences of general term $\prod_{j=0}^\ell \delta_j^{n^j}$ and $\sum_{i=1}^k \prod_{j=0}^\ell \xi_{i,j}^{n^j}$, respectively. Then $s_{n} = w_{n} \tilde{s}_{n}$, and hence $\omega(s_{n}) \ge \omega(\tilde{s}_{n})$. On the other hand, there cannot exist $\tilde{a}_0, \ldots, \tilde{a}_\ell \in \ZZb$ such that $\tilde{s}_{2n} = \prod_{j=0}^\ell \tilde{a}_j^{(2n)^j}$ for all $n$, or else we would have $s_{2n} = \prod_{j=0}^\ell (\delta_j \fixed[0.15]{\text{ }}\tilde{a}_j)^{(2n)^j}$ for every $n$ (which is impossible). This leads to the claim.

With the above in mind, let $\mathcal{P}$ be the set of all (positive) prime divisors of $\mathfrak{z} := \prod_{i=1}^k \prod_{j=1}^\ell x_{i,j}$; observe that $\mathcal{P}$ is finite and non-empty, as the preceding considerations yield $|\mathfrak{z}| \ge 2$. Then
\begin{equation}
\label{equ:rewriting}
s_{n}=\sum_{i=1}^k \!\left(x_{i,0} \prod_{p \in \mathcal{P}} p^{e_p^{(i)}(n)}\right)\fixed[0.2]{\text{ }}\!\!,
\quad\text{for every }n \in \mathbf N^+,
\end{equation}
where
$e_p^{(i)}$ denotes, for all $p \in \mathbf P$ and $i \in \llb 1, k \rrb$, the function
$
\NNb^+ \to \NNb: n \mapsto \sum_{j=1}^\ell n^j \upsilon_p(x_{i,j})$.

Since $\delta_0 = \cdots = \delta_\ell = 1$, it is easily seen that for every $p \in \mathbf P$ there are $i,j \in \llb 1, k \rrb$ for which $e_p^{(i)} \ne e_p^{(j)}$, and there exist $i_p \in \llb 1, k \rrb$ and $n_p \ge N$ such that $e_p^{(i_p)}(n) < e_p^{(i)}(n)$ for all $n \ge n_p$ and $i \in \llb 1, k \rrb$ for which $e_p^{(i)} \ne e_p^{(i_p)}$.
Let $n_\PPc := \max_{p \in \mathcal P} n_p$ (recall that $\mathcal P$ is a non-empty finite set), and for each $p \in \mathcal P$ and $i \in \llb 1, k \rrb$ define $
\Delta e_p^{(i)} := e_p^{(i)} - e_p^{(i_p)}$. Then set
\begin{equation}
\label{equ:reduced_sequence}
\pi_n := \prod_{p \in \mathcal{P}} p^{e_p^{(i_p)}(n)}
\quad\text{and}\quad
\sigma_n := \sum_{i=1}^k \!\left(x_{i,0} \prod_{p \in \mathcal{P}} p^{\Delta e_p^{(i)}(n)}\right)\fixed[0.2]{\text{ }}\!\!.
\end{equation}
We have $|s_{n}| = \pi_{n} \cdot |\sigma_{n}|$, and we obtain from \eqref{equ:subsums} that $\sigma_n \in \mathbf Z \setminus \{0\}$ for $n \ge n_\PPc$.
Further\-more, having assumed $x_{i,j} > 0$ for all $(i,j) \in \llb 1, k \rrb \times \llb 1, \ell \rrb$ implies, together with \eqref{equ:ordered} and \eqref{equ:rewriting}, that
\begin{equation}
\label{equ:limit}
\lim_{n \to \infty} \prod_{p \in \mathcal P} p^{e_p^{(k)}(n) - e_p^{(i)}(n)} =
\lim_{n \to \infty} \prod_{j=1}^\ell \left(\frac{x_{k,j}}{x_{i,j}}\right)^{\!n^j} = \infty, \quad\text{for each }i \in \llb 1, k-1 \rrb.
\end{equation}
Consequently, we find that
\begin{equation}
\label{equ:asymptotic_for_s2n}
|s_n| \sim |x_{k,0}| \cdot \prod_{j=1}^\ell x_{k,j}^{n^j} = |x_{k,0}| \cdot \prod_{p \in \mathcal P} p^{e_p^{(k)}(n)},\quad\text{as }n \to \infty
\end{equation}
and
\begin{equation}
\label{equ:asymptotic_for_sigma2n}
|\sigma_{n}| = \frac{|s_{n}|}{\pi_{n}} \sim |x_{k,0}| \cdot \prod_{p \in \mathcal P} p^{\Delta e_p^{(k)}(n)},\quad\text{as }n \to \infty.
\end{equation}
We want to show that the sequence $(|\sigma_{n}|)_{n \ge 1}$ is eventually (strictly) increasing.
\begin{lemma}
\label{lem:1}
There exists $p \in \PPc$ such that $\Delta e_p^{(k)}(n) \to \infty$ as $n \to \infty$.
\end{lemma}
\begin{proof}
Suppose the contrary is true. Then, for each $p \in \PPc$ we have $e_p^{(k)} = e_p^{(i_p)}$, since $\Delta e_p^{(k)}(n)$ is basically a polynomial with integral coefficients in the variable $n$ and $\Delta e_p^{(k)}(n) = e_p^{(k)}(n) - e_p^{(i_p)}(n) \ge 0$ for $n \ge n_\PPc$. Therefore, we get from \eqref{equ:limit} that
$$
\prod_{p \in \PPc} p^{e_p^{(i_p)}(n)} \le \prod_{p \in \PPc} p^{e_p^{(i)}(n)} \le \prod_{p \in \PPc} p^{e_p^{(k)}(n)} = \prod_{p \in \PPc} p^{e_p^{(i_p)}(n)}
$$
for all $n \ge n_\PPc$ and $i \in \llb 1, k \rrb$. But this is impossible, as it implies that $e_p^{(i)} = e_p^{(i_p)}$ for all $p \in \PPc$ and $i \in \llb 1, k \rrb$, and hence, in view of \eqref{equ:rewriting}, $s_{n} = (x_{1,0} + \cdots + x_{k,0}) \cdot \prod_{p \in \PPc} p^{e_p^{(i_p)}(n)}$ for all $n$.
\end{proof}
Now, let $A := 2 \mathfrak z^2$ (this is just a convenient value for $A$: We make no effort to try to optimize it, and the same is true for other constants later on). Since $\Delta e_p^{(k)}$ is eventually non-\-de\-creas\-ing for every $p \in \mathbf P$ (recall that $\Delta e_p^{(k)}$ is a polynomial function and $\Delta e_p^{(k)}(n) \ge 0$ for all large $n$), we obtain from \eqref{equ:subsums}, \eqref{equ:asymptotic_for_s2n}, \eqref{equ:asymptotic_for_sigma2n}, and Lemma \ref{lem:1} that there exists $n_0 \ge \max(2,n_\PPc)$ such that
\begin{equation}
\label{equ:growth_bound_and_monotonicity}
\sigma_{n}^{\fixed[0.2]{\text{ }}2} \le s_{n}^{\fixed[0.2]{\text{ }}2} < A^{n^\ell}
\quad\text{and}\quad
0 \ne |\sigma_{n}| < |\sigma_{n+1}|,
\quad \text{for }n \ge n_0.
\end{equation}
From here on, the proof of Case (i) splits, as we present two different approaches that can be used to finish it, the first of them relying on a theorem of Evertse from \cite{Ev95}, and the second using only elementary methods (as anticipated in the introduction).
\vskip 0.3cm
\noindent{}\framebox[1.1\width]{\textsc{1st Approach:}} Let $\mathfrak y := \mathfrak z \cdot \prod_{i=1}^k |x_{i,0}|$ and $B :=  \max(n_0^\ell,(2^{35} k^2)^{2 k^3 \mathfrak{y} \ell \log A})$. We will need the following:
\begin{lemma}\label{lem:ref}
There is a sequence $(r_\kappa)_{\kappa \ge 0}$ of integers $\ge n_0$ such that $r_\kappa^{\fixed[0.15]{\text{ }}\ell} \le B^{\otimes (\kappa+1)}$ and $\omega(s_{r_{\kappa}}) \ge \kappa$ for every $\kappa \in \mathbf N$.
\end{lemma}
\begin{proof}
Set $r_0 := n_0$, fix $\kappa \in \mathbf N^+$, and suppose we have already found an integer $r_{\kappa-1} \ge n_0$ such that $r_{\kappa-1}^{\fixed[0.15]{\text{ }}\ell} \le B^{\otimes \kappa}$ and $\omega(s_{r_{\kappa-1}}) \ge \kappa-1$: Notice how these conditions are trivially satisfied for $\kappa = 1$, because $r_0^{\fixed[0.15]{\text{ }}\ell} = n_0^{\fixed[0.15]{\text{ }}\ell} \le B = B^{\otimes 1}$ and $\omega(x) \ge 0$ for all $x \in \mathbf Z$.

Accordingly, denote by $\mathcal S_\kappa$ the set of prime divisors of $\mathfrak y \cdot s_{r_{\kappa-1}}$,
and for all $n \ge n_0$ and $i \in \llb 1, k \rrb$ let $X_i(n) := s_n^{-1} \cdot \prod_{j=0}^\ell x_{i,j}^{n^j}$ (note that these quantities are well defined, since we have by \eqref{equ:subsums} that $s_n \ne 0$ for $n \ge n_0$). A few remarks are in order.

Firstly, it is easy to check that, for every $n \ge n_0$, the $k$-tuple ${\bm X}_n := (X_1(n), \ldots, X_k(n)) \in \mathbf Q^k$ is a so\-lu\-tion to the following equation (over the additive group of the rational field):
\begin{equation}
\label{equ:unit-equation-in-Q}
Y_1 + \cdots + Y_k = 1,
\end{equation}
and we derive from \eqref{equ:subsums} that it is actually a \textit{non-degenerate} solution, where a solution $(Y_1, \ldots, Y_k)$ of \eqref{equ:unit-equation-in-Q} is called non-degenerate if $\sum_{i \in I} Y_i \ne 0$ for every non-empty $I \subseteq \llb 1, k \rrb$.

Secondly, it is plain from our definitions that
${\bm X}_{m} = {\bm X}_{n}$, for some $m, n \ge n_0$, only if
\begin{equation}
\label{equ:p-adic-equality}
\upsilon_p(X_i(m)) = \upsilon_p(X_i(n)),
\quad\text{for all }p \in \mathbf P
\text{ and }
i \in \llb 1, k \rrb,
\end{equation}
and we want to show that this, in turn, is possible only if
$|\sigma_m| = |\sigma_n|$.

Indeed, let $p \in \mathbf P$ and $n \ge n_0$. By construction,
the $p$-adic valuation of $\prod_{j=1}^\ell x_{i_p,j}^{n^j}$ is equal to $e_p^{(i_p)}(n)$, with $e_p^{(i_p)}(n)$ being zero if $p \notin \mathcal P$. Thus, we obtain from \eqref{equ:reduced_sequence} that
$$
\upsilon_p(X_{i_p}(n)) = \upsilon_p(x_{i_p,0}) + e_p^{(i_p)}(n) - \upsilon_p(s_{n}) = v_p(x_{i_p,0}) - \upsilon_p(\sigma_n).
$$
It follows that, for $m, n \ge n_0$, \eqref{equ:p-adic-equality} holds true only if $\upsilon_p(\sigma_m) = \upsilon_p(\sigma_n)$ for all $p \in \mathbf P$, which is equivalent to $|\sigma_m| = |\sigma_n|$. Accordingly, we conclude from \eqref{equ:growth_bound_and_monotonicity} and the above that, for $m, n \ge n_0$ and $m \ne n$, ${\bm X}_m$ and ${\bm X}_n$ are \textit{distinct} non-degenerate solutions of \eqref{equ:unit-equation-in-Q}.

Thirdly, let $N_\kappa$ be the number of non-degenerate solutions $(Y_1, \ldots, Y_k)$ to \eqref{equ:unit-equation-in-Q} for which each $Y_i$ is an $\mathcal S_\kappa$-unit (i.e., lies in the subgroup of the multiplicative group of $\mathbf Q$ generated by $\mathcal S_\kappa$). We obtain from \cite[Theorem 3]{Ev95} that $N_\kappa \le (2^{35} k^2)^{k^3 g_\kappa}$, where
\begin{equation*}
\begin{split}
g_\kappa := |\mathcal S_\kappa| & \le \omega(\mathfrak y) + \omega(s_{r_{\kappa-1}}) \le \mathfrak y + \log |s_{r_{\kappa-1}}|
\fixed[-0.65]{\text{ }}\stackrel{\eqref{equ:growth_bound_and_monotonicity}}{\le}\fixed[-0.65]{\text{ }}
\mathfrak y + r_{\kappa-1}^{\fixed[0.15]{\text{ }}\ell} \log A \le \mathfrak y \cdot r_{\kappa-1}^{\fixed[0.15]{\text{ }}\ell} \log A.
\end{split}
\end{equation*}
Using that $r_{\kappa-1}^{\fixed[0.15]{\text{ }}\ell} \le B^{\otimes \kappa}$ (by the inductive hypothesis), we thus conclude that
\begin{equation}
\label{equ:bound-on-number-of-nondegenerate-solutions-to-S-unit-equ}
N_{\kappa }\le (2^{35} k^{2})^{B^{\otimes \kappa } k^{3} \mathfrak{y}
 \log A} \le B
^{B^{\otimes \kappa }/(2\ell )}.
\end{equation}
With this in hand, define $\wp := \prod_{p \fixed[0.35]{\text{ }}\mid \fixed[0.3]{\text{ }} s_{r_{\kappa-1}}} (p-1)$ and let $(t_h)_{h \ge 0}$ be the subsequence of $(s_n)_{n \ge 1}$ of general term $t_h := s_{h\wp+r_{\kappa-1}}$. We know from the above that there exists $h_\kappa \in \llb 0, N_\kappa \rrb$ such that $t_{h_\kappa}$ is not an $\mathcal S_\kappa$-unit (note that $h\wp + r_{\kappa-1} \ge r_{\kappa-1} \ge n_0$ for all $h$), with the result that at least one prime divisor of $t_{h_\kappa}$ does not divide $s_{r_{\kappa-1}}$. On the other hand, a straightforward application of Fermat's little theorem shows that $p \mid t_{h_\kappa}$ for every $p \in \mathbf P$ such that $p \mid s_{r_{\kappa-1}}$. So, putting it all together, we find $\omega(s_{r_\kappa}) \ge 1 + \omega(s_{r_{\kappa-1}}) \ge \kappa$, where
\begin{equation*}
\label{equ:inductive-step-ineq}
\begin{split}
r_\kappa
    &
    := h_\kappa \wp + r_{\kappa-1} \le N_\kappa \fixed[0.25]{\text{ }} s_{r_{\kappa-1}} +r_{\kappa-1}
    \le N_\kappa \fixed[0.25]{\text{ }} s_{r_{\kappa-1}} \fixed[0.25]{\text{ }} r_{\kappa-1}
    \fixed[-0.75]{\text{ }}\stackrel{\eqref{equ:growth_bound_and_monotonicity}}{\le}\fixed[-0.75]{\text{ }}
    A^{r_{\kappa-1}^{\fixed[0.15]{\text{ }}\ell}} N_\kappa \fixed[0.25]{\text{ }} r_{\kappa-1} \le A^{2r_{\kappa -1}^{\ell }}
N_{\kappa } \\
    &
    \stackrel{\eqref{equ:bound-on-number-of-nondegenerate-solutions-to-S-unit-equ}}{\le}\fixed[-0.75]{\text{ }}
    A^{2r_{\kappa -1}^{\ell }} B^{B^{\otimes \kappa }/(2\ell )}
\le A^{2B^{
\otimes \kappa }}B^{B^{\otimes \kappa }/(2\ell )}  \le B^{B^{\otimes \kappa }/(2\ell )} B^{B^{\otimes
\kappa }/(2\ell )}
\end{split}
\end{equation*}
(recall from the above that $r_{\kappa-1}^{\fixed[0.25]{\text{ }}\ell} \le B^{\otimes \kappa}$). This completes the induction step (and hence the proof of the lemma), since it implies
$r_\kappa^{\fixed[0.25]{\text{ }}\ell} \le B^{\frac{3}{4} B^{\otimes \kappa}} B^{\otimes \kappa} \le B^{\otimes (\kappa+1)}$.
\end{proof}
So to conclude, let $(r_\kappa)_{\kappa \ge 0}$ be the sequence of Lemma \ref{lem:ref} and take $\theta :=B^{\otimes 3}$. Then $\theta > 1$ and
$\omega (s_{r_{\kappa }}) \ge \kappa > \mathrm{slog}_{\theta }(r_{
\kappa })$ for all $\kappa \in \mathbf{N}^{+}$, because $r_{\kappa }
\le r_{\kappa }^{\ell } \le B^{\otimes (\kappa
+1)}$.
\vskip 0.3cm
\noindent{}\framebox[1.1\width]{\textsc{2nd Approach:}} Denote by $\mathcal Q_n$ the set of all prime divisors of $\sigma_n$ and let $\mathcal Q_n^\star := \mathcal Q_n \setminus \mathcal P$. It is clear that $\mathcal Q_n$ is finite for $n \ge n_0$ (recall that $\sigma_n \ne 0$ for $n \ge n_\PPc$). Thus, let
$$
\lambda := \max_{p \in \mathcal P} \upsilon_p(\sigma_{n_0}) + \max_{p \in \mathcal P} \max_{1 \le i \le k} \Delta e_p^{(i)}(n_0),
$$
and then
$$
\alpha := k \cdot \max_{1 \le i \le k} |x_{i,0}| \cdot \prod_{p \in \mathcal{P}} p^\lambda,
\quad
\beta := \prod_{p \in \mathcal{P}} p^{\alpha-1}(p-1),
\quad\text{and}\quad
B := A^2 \beta.
$$
Lastly, suppose that, for a fixed $\kappa \in \NNb$, we have already found $r_0, \ldots, r_\kappa \in \mathbf N^+$ with $n_0 \le r_0 \le \cdots \le r_\kappa$, and define
$\beta_\kappa := \beta \cdot \prod_{p \in \mathcal{Q}_{r_\kappa}^\star} p^{\upsilon_p(\sigma_{r_\kappa})}(p-1)$.

By taking $r_0 := n_0$ and $r_{\kappa+1} := \beta_\kappa + r_\kappa$, we obtain an increasing sequence $(r_\kappa)_{\kappa \ge 0}$ of integers $\ge n_0$ with the property that, however we choose a prime $p \in \PPc$ and an index $i \in \llb 1, k \rrb$,
\begin{equation}
\label{equ:congruence_mod_P_in_the_exp}
\Delta e_p^{(i)}(r_{\kappa+1}) \equiv \Delta e_p^{(i)}(r_\kappa) \bmod q^{\alpha-1}(q-1),\quad\text{for all }q \in \mathcal P
\end{equation}
and
\begin{equation}
\label{equ:congruence_mod_Qi_in_the_exp}
\Delta e_p^{(i)}(r_{\kappa+1}) \equiv \Delta e_p^{(i)}(r_\kappa) \bmod q^{\upsilon_q(\sigma_{r_\kappa})}(q-1),\quad\text{for all }q \in \QQc_{r_\kappa}^\star,
\end{equation}
where we use that $\Delta e_p^{(i)}$ is essentially a polynomial with integral coefficients, and $r_{\kappa+1} \equiv r_\kappa \bmod m$ whenever $m \mid \beta_\kappa$. In particular,  \eqref{equ:congruence_mod_P_in_the_exp} and a routine induction imply
\begin{equation}
\label{equ:reduction_ad_zerum}
\Delta e_p^{(i)}(r_\kappa) \equiv \Delta e_p^{(i)}(n_0) \bmod q^{\alpha-1}(q-1),\quad\text{for all }p, q \in \mathcal P,\ i \in \llb 1, k \rrb, \text{ and }\kappa \in \NNb.
\end{equation}
Also, since $r_\kappa \ge n_0$, there exists $B > A$ such that, for all $\kappa$,
\begin{equation}
\label{equ:upper_bound_on_beta_k^2}
r_{\kappa+1} \le r_\kappa + \beta \cdot \prod_{p \in \mathcal{Q}_{r_\kappa}} p^{\upsilon_p(\sigma_{r_\kappa})}(p-1) \le r_\kappa + \beta \fixed[0.15]{\text{ }}\sigma_{r_\kappa}^2 <
r_\kappa + \beta A^{r_\kappa^\ell} \le \beta r_\kappa A^{r_\kappa^\ell}
\fixed[-0.35]{\text{ }}\stackrel{\eqref{equ:growth_bound_and_monotonicity}}{<}\fixed[-0.35]{\text{ }}
B^{\fixed[0.1]{\text{ }} r_\kappa^\ell}.
\end{equation}
Based on these premises, we prove a series of three lemmas. To ease notation, we denote by $I_p$, for each $p \in \PPc$, the set of all $i \in \llb 1, k \rrb$ such that $e_p^{(i)} \ne e_p^{(i_p)}$, and we let $I_p^\star := \llb 1, k \rrb \setminus I_p$.
\begin{lemma}
\label{lem:2}
$\mathcal Q_{r_\kappa} \subseteq \mathcal Q_{r_{\kappa+1}}$ for every $\kappa$.
\end{lemma}
\begin{proof}
Pick $\kappa \in \NNb$ and $q \in \mathcal Q_{r_\kappa}$.
If $i \in I_p$, then $\Delta e_p^{(i)}(n) = 0$ for all $n$, and hence $p^{\Delta e_p^{(i)}(n)} = 1$. If, on the other hand, $i \in I_p^\star$, then $\Delta e_p^{(i)}(n) > 0$ for $n \ge n_\PPc$, and we conclude from Fermat's little theorem that $p^{\Delta e_p^{(i)}(n)} \equiv 0 \bmod q$ if $q = p$, and $p^{\Delta e_p^{(i)}(n)} \equiv p^m \bmod q$ if $p \ne q$ and $\Delta e_p^{(i)}(n) \equiv m \bmod (q-1)$. So we get from \eqref{equ:congruence_mod_P_in_the_exp}, \eqref{equ:congruence_mod_Qi_in_the_exp}, and $r_{\kappa+1} > r_{\kappa} \ge n_0 > n_\PPc$ that
$$
p^{\Delta e_p^{(i)}(r_{\kappa+1})} \equiv p^{\Delta e_p^{(i)}(r_\kappa)} \bmod q,\quad\text{for all }p \in \PPc\text{ and }i \in \llb 1, k \rrb,
$$
which in turn implies
$$
\sigma_{r_{\kappa+1}} \equiv \sum_{i=1}^k \!\left(x_{i,0} \prod_{p \in \mathcal{P}} p^{\Delta e_p^{(i)}(r_{\kappa+1})}\right)\! \equiv \sum_{i=1}^k \!\left(x_{i,0} \prod_{p \in \mathcal{P}} p^{\Delta e_p^{(i)}(r_\kappa)}\right)\! \equiv \sigma_{r_\kappa} \equiv 0 \bmod q.
$$
This finishes the proof, since $\kappa\in {\bf N}$ and $q\in {\mathcal Q}_{r_{\kappa}}$ were arbitrary.
\end{proof}
\begin{lemma}
\label{lem:3}
Let $q \in \PPc$ and $\kappa \in \NNb$. Then $\upsilon_q(\sigma_{r_\kappa}) \le \alpha - 1$.
\end{lemma}
\begin{proof}
The claim is straightforward if $\kappa = 0$, since $r_0 = n_0$ and $\upsilon_q(\sigma_{n_0}) \le \lambda < \alpha$. So assume for the rest of the proof that $\kappa \ge 1$.
Then, we have from \eqref{equ:reduced_sequence} that
\begin{equation}
\label{equ:decomposition_of_a_sum}
\sigma_n =
\sum_{i \in I_q}\!\left(x_{i,0} \prod_{p \in \mathcal{P}} p^{\Delta e_p^{(i)}(n)}\right)\! + \sum_{i \in I_q^\star}\!\left(x_{i,0} \prod_{q \ne p \in \mathcal{P}} p^{\Delta e_p^{(i)}(n)}\right)\fixed[-0.8]{\text{ }},\quad\text{for all }n.
\end{equation}
If $i \in I_q$, $n > n_0$ and $n \equiv n_0 \bmod \beta$, then $q^\alpha$ divides $\prod_{p \in \mathcal P} p^{\Delta e_p^{(i)}(n)}$, because $n \mid \Delta e_p^{(i)}(n)$ and $\Delta e_p^{(i)}(n) \ne 0$, hence $\alpha < \beta < \beta + n_0 \le n \le \Delta e_p^{(i)}(n)$.

On the other hand, it is seen by induction that $r_\kappa \equiv n_0 \bmod \beta$ (recall that $r_{\kappa} \equiv r_{\kappa-1} \bmod \beta$). Thus, we get from the above, equations \eqref{equ:decomposition_of_a_sum} and \eqref{equ:reduction_ad_zerum}, \cite[Theorem 2.5(a)]{Apo}, and Euler's totient theorem that
\begin{equation}
\label{equ:reduction_ii}
\sigma_{r_{\kappa}} \equiv \sum_{i \in I_q^\star}\!\left(x_{i,0} \prod_{q \ne p \in \mathcal{P}} p^{\Delta e_p^{(i)}(r_{\kappa})}\right)\!
\equiv \sum_{i \in I_q^\star}\!\left(x_{i,0} \prod_{q \ne p \in \mathcal{P} } p^{\Delta e_p^{(i)}(n_0)}\right)\!
\bmod q^{\alpha}.
\end{equation}
But $\emptyset \ne I_q^\star \subseteq \llb 1, k \rrb$, so it follows from \eqref{equ:subsums} that
\begin{equation*}
\begin{split}
0 & < \left|\frac{1}{\pi_{n_0}} \sum_{i \in I_q^\star} \prod_{j=0}^\ell x_{i,j}^{(n_0)^j}\right|
     = \left|\sum_{i \in I_q^\star} \!\left(x_{i,0} \prod_{q \ne p \in \mathcal{P}} p^{\Delta e_p^{(i)}(n_0)}\right) \right| \\
    & \le \max_{1 \le i \le k} |x_{i,0}| \cdot \sum_{i \in I_q^\star} \prod_{q \ne p \in \mathcal{P}} p^{\Delta e_p^{(i)}(n_0)} \le k \cdot \max_{1 \le i \le k} |x_{i,0}| \cdot \prod_{p \in \mathcal{P}} p^{\lambda} = \alpha < q^{\alpha},
\end{split}
\end{equation*}
which, together with \eqref{equ:reduction_ii}, yields $\upsilon_q(\sigma_{r_{\kappa}}) < \alpha$.
\end{proof}
\begin{lemma}
\label{lem:4}
Let $\kappa \in \NNb^+$ and $q \in \QQc_{r_\kappa}$. Then $\upsilon_q(\sigma_{r_\kappa}) = \upsilon_q(\sigma_{r_{\kappa+1}})$.
\end{lemma}
\begin{proof}
If $q \notin \PPc$, then we infer from \eqref{equ:reduced_sequence} and \eqref{equ:congruence_mod_Qi_in_the_exp}, \cite[Theorem 2.5(a)]{Apo}, and Euler's totient theorem that
$\sigma_{r_{\kappa+1}} \equiv \sigma_{r_\kappa} \bmod q^{\upsilon_q(\sigma_{r_\kappa}) + 1}$, and we are done.

If, on the other hand, $q \in \PPc$, then we get from Lemma \ref{lem:3} that $\upsilon_q(\sigma_{r_1}) \le \alpha-1$, which, along with \eqref{equ:reduction_ii}, gives $
\sigma_{r_\kappa} \equiv \sigma_{r_1} \bmod q^{\upsilon_q(\sigma_{r_1})+1}$, and consequently $\upsilon_q(\sigma_{r_\kappa}) = \upsilon_q(\sigma_{r_1})$.
\end{proof}
At this point, since $(|\sigma_n|)_{n \ge n_0}$ is an increasing sequence by \eqref{equ:growth_bound_and_monotonicity} and $r_\kappa \ge n_0$ for all $\kappa \in \mathbf N^+$, we see from Lemmas \ref{lem:2}-\ref{lem:4} that $\emptyset \ne \QQc_{r_\kappa} \subsetneq \QQc_{r_{\kappa+1}}$, and hence $\omega(\sigma_{r_\kappa}) < \omega(\sigma_{r_{\kappa+1}})$. By induction, this implies $\omega(\sigma_{r_\kappa}) \ge \kappa$ for every $\kappa$.

On the other hand, if we let $\theta := \max(B^\ell \ell, r_1^{\fixed[0.1]{\text{ }}\ell})$, then we get from \eqref{equ:upper_bound_on_beta_k^2} and another induction that $r_\kappa^{\fixed[0.1]{\text{ }}\ell} < \theta^{\otimes\kappa}$ for all $\kappa \in \NNb^+$, which, together with the above considerations, leads to $\omega(\sigma_{r_\kappa}) \ge \kappa > \slog_\theta(r_\kappa)$ and the desired conclusion.
\vskip 0.3cm
\noindent{}\textbf{Case (ii):} \textit{There do not exist $b_0, \ldots, b_\ell \in \QQb$ such that $s_{2n-1} = \prod_{j=0}^\ell b_j^{(2n-1)^j}$ for all $n$.} Then, we are reduced to Case (i) by taking
$$
y_{i,j} := \prod_{h=j}^\ell x_{i,h}^{(-1)^{h-j} \binom{h}{j}}, \quad \text{for }1 \le i \le k \text{ and }0 \le j \le \ell,
$$
and by noting that for every $n \in \NNb^+$ we have $s_{2n-1} = t_{2n}$, where $(t_n)_{n \ge 1}$ is the integer sequence of general term $\sum_{i=1}^k \prod_{j=0}^\ell y_{i,j}^{n^j}$
(we omit further details).
\section{Proof of Corollary \texorpdfstring{\ref{corollary}}{3}}
\label{sec:proof:cor(3)}
Suppose for a contradiction that there are $c_1, \ldots, c_k \in \QQb^+$ and $x_1, \ldots, x_k \in \mathbf Q \setminus \{0\}$ such that $|x_i| \ne |x_j|$ for some $i, j \in \llb 1, k \rrb$ and $(\omega(u_n))_{n \ge 1}$ is bounded, where $u_n := \sum_{i=1}^k c_i x_i^n$ for all $n$, and let $k$ the \textit{minimal} positive integer for which this is pretended to be true.

Then $k \ge 2$, and we can assume that $|x_1| \le \cdots \le |x_k| \ne |x_1|$. Furthermore, we get from Theorem \ref{th:main} that there must exist $c, x \in \mathbf Q^+$ such that $u_{2n} = c \fixed[0.2]{\text{ }} x^{2n}$. So now, we have two cases, each of which will lead to a contradiction (the rest is trivial and we may omit details):
\vskip 0.3cm
\noindent{}\textbf{Case (i):} $x \le |x_k|$. We have $c\fixed[0.2]{\text{ }} y^{2n} = \sum_{i=1}^k c_i y_i^{2n}$ for all $n$, where $y_i := |x_i| \cdot |x_k|^{-1}$ for $1 \le i \le k$ and $y := x \cdot |x_k|^{-1}$. Let $h$ be the maximal index in $\llb 2, k \rrb$ such that $y_{h-1} < y_k$, which exists because $y_1 < y_k$. Since $0 < y \le 1$ and $0 < y_i < 1$ for $1 \le i < h$, we find that
$$
c \cdot \lim_{n \to \infty} y^{2n} = c_h + \cdots + c_k,
$$
which can happen only if $y = 1$, as $c_h, \ldots, c_k > 0$. But then $c = c_1 + \cdots + c_k$, and consequently $\sum_{i=1}^{h-1} c_i y_i^{2n} = 0$ for all $n$, which is impossible, because $h \ge 2$ and $c_1, \ldots, c_{h-1} > 0$.
\vskip 0.3cm
\noindent{}\textbf{Case (ii):} $x > |x_k|$. Then $c = \sum_{i=1}^k c_i z_i^{2n}$ for all $n$, where $z_i := |x_i| \cdot x^{-1}$ for $1 \le i \le k$. But this is still impossible, since $z_1, \ldots, z_k \in {]0,1[}\fixed[0.15]{\text{ }}$, and hence $\sum_{i=1}^k c_i z_i^{2n} \to 0$ as $n \to \infty$.
\section{Closing remarks}
\label{sec:closings}
Let $\tau$ be an increasing function from $\NNb^+$ into itself. What can be said about the behavior of $\omega(s_{\tau(n)})$ as $n \to \infty$?
And what about the asymptotic growth of the average of the function $\RRb^+ \to \NNb: x \mapsto \#\{n \le x: \omega(s_{\tau(n)}) \ge h\}$ for a fixed $h \in \NNb^+$?

In this paper, we have answered the first question in the case where $\tau$ is the identity or, more in general, a polynomial function (by the considerations made in the introduction). It could be interesting as a next step to look at the case where $\tau$ is a geometric progression, which however may be hard, when taking into account that it is a longstanding open problem to decide whether there are infinitely many composite Fermat numbers (that is, numbers of the form $2^{2^n} + 1$).

On the other hand, the basic question addressed in the present manuscript has
the following algebraic generalization:
Given a unique factorization domain $D$, let $\alpha_{i,j}$ be, for $1 \le i \le k$ and $0 \le j \le \ell$, some fixed elements in $D$, and for $x \in D$ let $\omega_D(x)$ denote the number of non-associate primes dividing $x$. What can be said about the sequence $(A_n)_{n \ge 1}$ of general term
$
\sum_{i=1}^k \prod_{j=0}^\ell \alpha_{i,j}^{\fixed[0.1]{\text{ }} n^j}
$
if the sequence $(\omega_D(A_n))_{n \ge 1}$ is bounded? Does anything along the lines of Theorem \ref{th:main} hold true?
\section*{Acknowledgements}
\label{subsec:acks}
The authors are grateful to Carlo Sanna (Universit\`a di Torino, IT) for useful comments. Moreover, they are greatly indebted to an anonymous referee for outlining a proof of Lemma \ref{lem:ref} in the case when the numbers $x_{i,j}$ are positive, and most notably for explaining how to use \cite[Theorem 3]{Ev95} to bound, in the notations of the same proof, the number of non-degenerate $\mathcal S_\kappa$-unit solutions of equation \eqref{equ:unit-equation-in-Q}, which is the key step in the first approach (``Approach 1'') to the proof of Theorem \ref{th:main}.
%
%

%
\end{document}